\documentclass[12pt,a4paper]{article}
\pdfoutput=1 
\usepackage[T1]{fontenc}
\usepackage[utf8]{inputenc}
\usepackage[english]{babel}
\usepackage{authblk}
\usepackage{amsmath,amssymb,amsthm}
\usepackage{graphicx}
\usepackage{color}
\usepackage{hyperref}
\usepackage{todonotes}
\usepackage{amsfonts}
\usepackage{mathrsfs}
\usepackage{url}
\newtheorem{theorem}{Theorem}
\newtheorem{lemma}{Lemma}
\newtheorem{corollary}{Corollary}
\newtheorem{remark}{Remark}

\allowdisplaybreaks

\numberwithin{equation}{section}

\newtheoremstyle{myplain}
{3pt}
{3pt}
{\itshape}
{\parindent}
{\bfseries}
{.}
{.5em}
{}

\newcommand{\Z}{{\mathbf Z}}
\newcommand{\N}{{\mathbf N}}
\newcommand{\R}{{\mathbf R}}

\newcommand{\Expect}{\mathsf{E}}

\let\le=\leqslant

\let\ge=\geqslant

\makeatletter
\def\@maketitle{%
  \newpage
  \null
  \begin{center}%
  {\large\bf
    \lineskip .5em%
    \begin{tabular}[t]{c}%
      \@author
    \end{tabular}\par}%
  \vskip 2em%
  \let \footnote \thanks
    \vskip 1.5em%
    {\large\bf \@title \par}%
    \vskip 1em%
    {\normalfont \@date}%
  \end{center}%
  \par
  \vskip 1.5em}

\renewcommand\section{\@startsection {section}{1}{\parindent}%
                                   {2.5ex \@plus 1ex \@minus .2ex}%
                                   {-2.3ex \@plus -.2ex}%
                                   {\normalfont\bf}}
\renewcommand{\@seccntformat}[1]{\csname the#1\endcsname.~}
\renewcommand\@biblabel[1]{#1.}
\makeatother

\begin{document}


\date{}

\begin{center}
	\textbf{{A Limit Theorem for Supercritical Branching Random Walks with Branching Sources of Varying Intensity}}
\end{center}

\begin{center}
	\textit{{\large Ivan Khristolyubov, Elena Yarovaya}}
\end{center}

\begin{abstract}
We consider a supercritical symmetric continuous-time branching random walk
on a multidimensional lattice with a finite number of particle generation
sources of varying positive intensities without any restrictions on the variance of
jumps of the underlying random walk. It is assumed that the spectrum of the
evolution operator contains at least one positive eigenvalue. We prove that
under these conditions the largest eigenvalue of the evolution operator is
simple and determines the rate of exponential growth of particle quantities
at every point on the lattice as well as on the lattice as a whole.

\end{abstract}


\section{Introduction.}

We consider a continuous-time branching random walk (BRW) with a finite
number of branching sources that are situated at some points
$x_{1},x_{2},\dots,x_{N}$ of the lattice $\Z^{d}$, $d\ge1$,
see~\cite{Y17_TPA:e} for details.  The behaviour of BRWs, which are based on
symmetric spatially homogeneous irreducible random walks on $\Z^{d}$ with
finite variance of jumps, for the case of a single branching source was
considered, for example, in~\cite{YarBRW:e}. To the authors' best knowledge,
BRWs with a finite variance of jumps and a finite number of branching sources
of various types, at some of which the underlying random walk can become
asymmetric, were first introduced in~\cite{Y12_MZ:e}, and BRWs with identical
branching sources and no restrictions on the variance of jumps were first
considered in~\cite{Y16-MCAP:e}.

Let $\mu_{t}(y)$ be the number of particles at the time $t$ at the point $y$
under the condition that at the initial time $t=0$ the lattice $\Z^{d}$ contains a
single particle which is situated at $x$, that is, $\mu_{0}(y)=\delta(x-y)$. We
denote by $\mu_{t}=\sum_{y\in \Z^{d}}\mu_{t}(y)$ the total number of
particles on $\Z^{d}$. Let $m_{1}(t,x,y):=\Expect_{x}\mu_{t}(y)$  and
$m_{1}(t,x):=\Expect_{x}\sum_{y\in \Z^{d}}\mu_{t}(y)$ denote the expectation
of the number of particles at $y$ and on the lattice $\Z^{d}$ respectively
under the condition that $\mu_{0}(y)\equiv \delta(x-y)$ at the time $t$.

We assume the branching process at each of the branching sources
$x_{1},x_{2},\dots,x_{N}$ to be a continuous-time Galton-Watson process  (see
\cite[Ch.~I, \S4]{Se2:e}, \cite[Ch.~III]{AN:e}) defined by its infinitesimal
generating function  (which depends on the source $x_{i}$)
\begin{equation}\label{ch6:E-defGenFunc}
f(u,x_{i})=\sum_{n=0}^\infty b_n(x_{i}) u^n,\quad 0\le u \le1,
\end{equation}
where $b_n(x_{i})\ge0$ if $n\ne 1$, $b_1(x_{i})<0$, and $\sum_{n} b_n(x_{i})=0$.
We also assume that the inequalities $\beta_{i}^{(r)}:=f^{(r)}(1,x_{i})<\infty$ hold for all $i=1,2,\dots,N$ and $r\in{\N}$. We call
\begin{equation}\label{D:beta}
\beta_{i}:=\beta_{i}^{(1)}=f^{(1)}(1,x_{i})=\sum_{n}nb_{n}(x_{i})
\end{equation}
the \emph{intensity} of the branching source $x_{i}$.

The behaviour of $m_{1}(t,x,y)$ and $m_{1}(t,x,y)$ can be described in terms
of the \emph{evolution operator}
$\mathscr{H}:=\mathscr{H}_{\beta_{1},\ldots,\beta_{N}}$ \cite{Y17_TPA:e}, the
definition of which is recalled in Section~\ref{S:statement}. We call a BRW
\emph{supercritical}
 if the spectrum of the operator $\mathscr{H}$ contains at least one eigenvalue $\lambda>0$. In the case of a supercritical BRW with equal branching source intensities
$\beta_{1}=\beta_{2}=\dots=\beta_{N}$ with no restrictions on the variance of
jumps, it was shown in~\cite{Y16-MCAP:e} that the spectrum of $\mathscr{H}$
is real and contains no more than $N$ positive eigenvalues counted with their
multiplicity, and that the largest eigenvalue  $\lambda_{0}$ has multiplicity
$1$. In the present study the aforementioned result is extended to the case
of a supercritical BRW with positive source intensities
$\beta_{1},\beta_{2},\ldots,\beta_{N}$ with no restrictions on the
variance of jumps or the number of descendants particles can produce. The main result of this work is the following limit theorem, the
proof of which is provided in Section~\ref{S:limthm}.
\begin{theorem}\label{thm1} Let the operator $\mathscr{H}$ have an isolated eigenvalue $\lambda_{0}>0$, and let the remaining part of its spectrum be located on the halfline $\{\lambda\in\R:~\lambda\leqslant\lambda_{0}-\epsilon\}$, where
$\epsilon>0$. If $\beta_{i}^{(r)} = O(r! r^{r-1})$  for all $i = 1, \ldots,
N$ and $r\in\N$, then in the sense of convergence in distribution the
following statements hold:
\begin{equation}
\label{eq1}
\lim_{t \to \infty} \mu_{t}(y) e^{-\lambda_{0} t} = \psi(y)\xi,\quad \lim_{t \to \infty} \mu_{t} e^{-\lambda_{0} t} = \xi,
\end{equation}
where $\psi(y)$ is a non-negative non-random function and $\xi$ is a proper random variable.
\end{theorem}
Theorem~\ref{thm1} generalizes the results obtained in \cite{BY-2:e,YarBRW:e}
for a supercritical BRW on $\Z^{d}$ with finite variance of jumps and a
single branching source. Its proof is fundamentally based on Carleman's
condition~\cite[Th.~1.11]{ShT:e}. In the case of a single branching source
and particles producing no more than two descendants Theorem~\ref{thm1} was
proved in~\cite{YarBRW:e}. In the case of a single branching source and no
restrictions on the number of descendants particles can produce
Theorem~\ref{thm1} was provided in~\cite{BY-2:e} without proof.

Let us briefly outline the structure of the paper. In
Section~\ref{S:statement} we recall the formal definition of a BRW. In
Section~\ref{S:maineq} we provide some key evolution equations for generating
functions and the moments of particle quantities in the case of a BRW with
several branching sources (Theorems~\ref{thm01}--\ref{thm04}). These theorems
are a natural generalization of the corresponding results that were obtained
for BRWs with a single branching source in~\cite{YarBRW:e}. In
Section~\ref{S:WSBRW} we establish a criterion for the existence of positive
eigenvalues in the spectrum of $\mathscr{H}$ (Theorem~\ref{thm002}), which is
later used to examine the properties of the spectrum of this evolution
operator. We then prove Theorem~\ref{thm06} on the behaviour of particle
quantity moments. Section~\ref{S:limthm} is dedicated to the proof of
Theorem~\ref{thm1}.

\section{The BRW model.}\label{S:statement}

By a branching random walk (BRW) we mean a stochastic process that
combines a random walk of particles with their branching (birth or death) at
certain points on $\Z^{d}$ called \emph {branching sources}.
Let us give more precise definitions.

We assume that the random walk is defined by its matrix of transition
intensities $A=\left(a(x,y)\right)_{x,y \in \Z^{d}}$ that satisfies the
\emph{regularity property} $\sum_{y \in \Z^{d}}a(x,y)=0$ for all $x$, where
$a(x,y)\ge 0$ for $x\neq y$ and $-\infty<a(x,x)<0$.

Suppose that at the moment $t=0$ there is a single particle on the lattice
that is situated at the point $x\in \Z^{d}$. Following the axiomatics
provided in~\cite[Ch.~III, \S2]{GS:e}, the probabilities $p(h,x,y)$ of a particle situated at $x\notin\{x_{1},x_{2},\dots,x_{N}\}$ to move to
an arbitrary point $y$ over a short period of time $h$ can be represented as
\begin{align*}
p(h,x,y)&=a(x,y)h+o(h)\qquad \text{for}\quad y\ne x,\\
p(h,x,x)&=1+a(x,x)h+o(h).
\end{align*}
It follows from these equalities, see, for example,~\cite[Ch.~III]{GS:e}, that the
transition probabilities $p(t,x,y)$ satisfy the following system of
differential-difference equations (called the \emph{Kolmogorov backward
equations}):
\begin{equation}\label{E:ptxy}
\frac{\partial p(t,x,y)}{\partial t}=\sum_{x'}a(x,x') p(t,x',y),\qquad
p(0,x,y)=\delta(x-y),
\end{equation}
where $\delta(\cdot)$ is the discrete Kronecker $\delta$-function on $\Z^{d}$.

The branching process at each of the sources $x_{1},x_{2},\dots,x_{N}$ is governed by the infinitesimal generating function~\eqref{ch6:E-defGenFunc}. Of particular interest to us are the source intensities~\eqref{D:beta}, which can be rewritten as follows:
\[
\beta_{i}=
(-b_{1}(x_{i}))\left(\sum_{n\neq1}n\frac{b_{n}(x_{i})}{(-b_{1}(x_{i}))}-1\right),
\]
where the sum is the average number of descendants a particle has at the source
$x_{i}$.


If at the moment $t=0$ a particle is located at a point different from the branching
sources, then its random walk follows the rules above. Therefore in order to
complete the description of its evolution we only have to consider a
situation combining both the branching process and the random walk, that is
to say, when the particle is at one of the branching sources
$x_{1},x_{2},\dots,x_{N}$.  In this case the possible outcomes that can
happen over a small period of time $h$ are the following: the particle will
either move to a point $y\neq x_{i}$ with the probability of
\[
p(h,x_{i},y)=a(x_{i},y)h+o(h),
\]
or will remain at the source and produce $n\neq1$ descendants with the
probability of
\[
p_{*}(h,x_{i},n)=b_{n}(x_{i})h+o(h)
\]
(we suppose, that the particle itself is included in these $n$ descendants;
therefore, if $n=0$ we say that the particle dies), or no change will happen
to the particle at all, which has the probability of
\[
1-\sum_{y\neq
x_{i}}a(x_{i},y)h-\sum_{n\neq 1}b_{n}(x_{i})h+o(h).
\]
As a result, the sojourn time of a particle at the source $x_{i}$ is
exponentially distributed with the parameter
$-(a(x_{i},x_{i})+b_{1}(x_{i}))$. Note that each new particle evolves
according to the same law independently of other particles.

As it was shown in~\cite{Y12_MZ:e}, \cite{Y13-PS:e}, the moments $m_{1}(t,x,y)$
and $m_{1}(t,x)$ satisfy the following equations:
\begin{align}\label{E:m1xy}
\frac{\partial m_{1}(t,x,y)}{\partial t}&=\sum_{x'}a(x,x') m_{1}(t,x',y)+
\sum_{i=1}^{N} \beta_{i}\delta(x-x_{i})m_{1}(t,x,y),\\\label{E:m1xy1}
\frac{\partial m_{1}(t,x)}{\partial t}&=\sum_{x'}a(x,x') m_{1}(t,x')+
\sum_{i=1}^{N} \beta_{i}\delta(x-x_{i})m_{1}(t,x)
\end{align}
with the initial values $m_{1}(0,x,y)=\delta(x-y)$ and $m_{1}(0,x)\equiv
1$ respectively.

Equations~\eqref{E:ptxy}--\eqref{E:m1xy1} are rather difficult to analyze,
and therefore we will from now on only consider BRWs that satisfy the following additional and quite natural
assumptions. First, we assume that the
intensities $a(x,y)$ are \emph{symmetric} and \emph{spatially homogeneous},
that is, $a(x,y)=a(y,x)=a(0,y-x)$. This allows us, for the sake of brevity,
to denote by $a(x-y)$ any of the three pairwise equal functions $a(x,y)$,
$a(y,x)$, $a(0,y-x)$, that is, $a(x-y):= a(x,y)=a(y,x)=a(0,y-x)$. Second, we
 assume that the random walk is \emph{irreducible}, which in terms of the
matrix $A$ means that it itself is irreducible: for any $z\in\Z^{d}$ there is
such a set of vectors $z_{1},\dots,z_{k}\in\Z^{d}$ that
$z=\sum_{i=1}^{k}z_{i}$ and $a(z_{i})\neq0$ for $i=1,\dots,k$.

One approach to analysing equations~\eqref{E:ptxy} and~\eqref{E:m1xy}
consists in treating them as differential equations in Banach spaces. In
order to apply this approach to our case, we introduce the operators
\[
(\mathscr{A} u)(x)=\sum_{x'}a(x-x') u(x'),\qquad
(\Delta_{x_{i}}u)(x)=\delta(x-x_{i})u(x),\quad i=1,\ldots,N.
\]
on the set of functions $u(x)$, $x\in\Z^{d}$.
We also introduce the operator
\begin{equation}\label{E:defHbbb}
\mathscr{H}:=\mathscr{H}_{\beta_{1},\ldots,\beta_{N}}=
\mathscr{A}+\sum_{i=1}^{N}\beta_{i}\Delta_{x_{i}}.
\end{equation}
for each set of source intensities $\beta_{1},\ldots,\beta_{N}$.
Let us note that all these operators can be regarded as linear continuous
operators in any of the spaces $l^{p}(\mathbf{Z}^{d})$, $p\in[1,\infty]$. We
also point out that the operator $\mathscr{A}$
is self-adjoint in $l^{2}(\mathbf{Z}^{d})$
\cite{Y12_MZ:e,Y13-PS:e,Y16-MCAP:e}.

Now, treating for each $t\ge0$ and each $y\in\Z^{d}$ the functions
$p(t,\cdot,y)$ and $m_{1}(t,\cdot,y)$ as elements of $l^{p}(\mathbf{Z}^{d})$
for some $p$, we can rewrite (see, for example,~\cite{Y12_MZ:e})~\eqref{E:ptxy}
and~\eqref{E:m1xy} as differential equations in $l^{p}(\mathbf{Z}^{d})$:
\begin{alignat*}{2}
\frac{d p(t,x,y)}{d t} &=(\mathscr{A}p(t,\cdot,y))(x),&\qquad p(0,x,y)&=\delta(x-y),\\
\frac{d m_{1}(t,x,y)}{d t} &=(\mathscr{H} m_{1}(t,\cdot,y))(x),&
\qquad m_{1}(0,x,y)&=\delta(x-y),
\end{alignat*}
and~\eqref{E:m1xy1} as a differential equation in
$l^{\infty}(\mathbf{Z}^{d})$:
\begin{equation*}
\frac{d m_{1}(t,x)}{d t}=(\mathscr{H} m_{1}(t,\cdot))(x),
\qquad m_{1}(0,x)\equiv 1.
\end{equation*}
Note that the asymptotic behaviour for large $t$ of the transition
probabilities $p(t,x,y)$, as well as of the mean particle numbers
$m_{1}(t,x,y)$ $m_{1}(t,x)$ is tightly connected with the spectral properties
of the operators $\mathscr{A}$ and $\mathscr{H}$ respectively.

It is convenient to express various properties of the transition
probabilities $p(t,x,y)$ in terms of Green's function, which can be defined
as the Laplace transform of the transition probability $p(t,x,y)$:
\[
G_\lambda(x,y):=\int_0^\infty e^{-\lambda t}p(t,x,y)\,dt,\qquad \lambda\ge 0,
\]
and can also be rewritten (see, for example,~\cite[\S~2.2]{YarBRW:e}) as follows:
\[
G_\lambda(x,y)=\frac{1}{(2\pi)^d} \int_{ [-\pi,\pi ]^{d}}
\frac{e^{i(\theta, y-x)}} {\lambda-\phi(\theta)}\,d\theta =\frac{1}{(2\pi)^d} \int_{ [-\pi,\pi ]^{d}}
\frac{\cos{(\theta, y-x)}} {\lambda-\phi(\theta)}\,d\theta,
\]
where $x,y\in\Z^{d}$, $\lambda \ge 0$, and $\phi(\theta)$ is the Fourier transform of the transition intensity $a(z)$:
\begin{equation}\label{E:Fourier}
\phi(\theta):=\sum_{z\in\mathbf{Z}^{d}}a(z)e^{i(\theta,z)}=\sum_{x \in \Z^{d}} a(x) \cos(x, \theta),\qquad \theta\in[-\pi,\pi]^{d}.
\end{equation}

The function $G_{0}(x,y)$ has a simple meaning for a (non-branching) random
walk: namely, it is equal to the mean amount of time a particle spends at
$y\in\Z^{d}$ as $t\to\infty$ under the condition that at the initial moment
$t=0$ the particle was at $x\in\Z^{d}$. Also, the asymptotic behaviour of the
mean numbers of particles $m_{1}(t,x,y)$ and $m_{1}(t,x)$ as $t\to\infty$ can
be described in terms of the function  $G_\lambda(x,y)$, see,
e.g.,~\cite{YarBRW:e}. Lastly, in~\cite{Y17_TPA:e} it was shown that the
asymptotic behaviour of a BRW depends strongly on whether $G_{0}: =
G_{0}(0,0)$ is finite.

\begin{remark}
The approach described in this section, based on interpreting BRW evolution
equations as differential equations in Banach spaces, is also applicable to a wide selection of problems,
notably to describing the evolution of higher particle number moments (see, e.g.,~\cite{YarBRW:e}, \cite{Y12_MZ:e}).
\end{remark}

\section{Key equations and auxiliary results.}\label{S:maineq}

Let us introduce the Laplace generating functions of the random variables  $\mu_{t}(y)$ and
$\mu_{t}$ for $z \geqslant 0$:
\[
F(z; t, x, y):= \Expect_{x} e^{-z \mu_{t}(y)},\qquad
F(z; t, x):= \Expect_{x} e^{-z \mu_{t}}.
\]
where $\Expect_{x}$ is the mean on condition $\mu_{0}(\cdot) =
\delta_{x}(\cdot)$.

The following four theorems are a result of an immediate generalization of the corresponding theorems in \cite{YarBRW:e} proved for BRWs with a single branching source; since the reasoning is
virtually the same, these theorems are presented here without proof.

\begin{theorem}
\label{thm01} For all $0 \leqslant z \leqslant \infty$ the functions $F(z; t,
x)$ and $F(z; t, x, y)$ are continuously differentiable with respect to $t$ uniformly with respect to $x,y\in\Z^{d}$. They also satisfy the inequalities  $0 \leqslant F(z; t, x), F(z;
t, x, y) \leqslant 1$ and are the solutions to the following Cauchy problems in
$l^{\infty}\left(\Z^{d} \right)$:
\begin{alignat}{2}\label{E:Fzt}
\frac{d F(z;t, \cdot)}{d t} &= \mathscr{A}F(z;t, \cdot) +
\sum_{j=1}^{N} \Delta_{x_{j}} f_{j} \left(F(z; t, \cdot) \right),
\qquad &F(z; 0, \cdot) &= e^{-z},\\
\label{E:Fzty}\frac{d F(z; t, \cdot, y)}{d t} &= \mathscr{A}F(z;t, \cdot, y) +
\sum_{j=1}^{N} \Delta_{x_{j}} f_{j} \left(F(z;t, \cdot, y) \right),
\qquad &F(z;0, \cdot, y) &= e^{-z \delta_{y}(\cdot)}.
\end{alignat}
\end{theorem}

Theorem~\ref{thm01} allows us to advance from analysing the BRW at hand to
considering the corresponding Cauchy problem in a Banach space instead. We
also note that, contrary to the single branching source case examined
in~\cite{YarBRW:e}, there is not one but several terms
$\Delta_{x_{j}}f_{j}(F)$ in the right-hand side of equations~\eqref{E:Fzt}
and~\eqref{E:Fzty}, $j=1,2,\dots,N$.

Let us set
\[
m_{n}(t, x, y) := \Expect_{x} \mu_{t}^{n}(y),\qquad
m_{n}(t, x) := \Expect_{x} \mu_{t}^{n}.
\]
\begin{theorem}
\label{thm02} For all natural $k \geqslant 1$ the moments $m_{k}(t, \cdot, y)\in
l^{2}\left(\Z^{d}\right)$ and $m_{k}(t, \cdot)\in
l^{\infty}\left(\Z^{d}\right)$ satisfy the following differential equations in the corresponding Banach spaces:
\begin{align}
\label{eq007}
\frac{d m_{1}}{d\,t} &= \mathscr{H}m_{1},\\
\label{eq008}
\frac{d m_{k}}{d\,t} &= \mathscr{H}m_{k} +
\sum_{j=1}^{N} \Delta_{x_{j}} g_{k}^{(j)} (m_{1}, \ldots, m_{k-1}),\qquad k \geqslant 2,
\end{align}
the initial values being $m_{n}(0, \cdot, y) = \delta_{y}(\cdot)$ and $m_{n}(0,
\cdot) \equiv 1$ respectively. Here $\mathscr{H}m_{k}$ stands for $\mathscr{H}m_{k}(t, \cdot, y)$ or $\mathscr{H}m_{k}
(t, \cdot)$ respectively, and
\begin{equation}\label{E:defgnj}
g_{k}^{(j)} (m_{1}, \ldots, m_{k-1}) :=\sum_{r=2}^{k} \frac{\beta_{j}^{(r)}}{r!} \sum_{\substack{i_{1}, \ldots, i_{r} > 0 \\ i_{1} + \cdots + i_{r} = n}} \frac{n!}{i_{1}! \cdots i_{r}!} m_{i_{1}} \cdots m_{i_{r}}.
\end{equation}
\end{theorem}
Theorem~\ref{thm02} will later be used in the proof of Theorem~\ref{thm06} to help determine the asymptotic behaviour of the moments as $t\to \infty$.

\begin{theorem}
\label{thm03} The moments $m_{1}(t, x,\cdot)\in l^{2}\left(\Z^{d}\right)$
satisfy the following Cauchy problem in $l^{2}\left(\Z^{d}\right)$:
\[
\frac{d m_{1}(t, x,\cdot)}{d t} = \mathscr{H} m_{1}(t, x,\cdot),\qquad m_{1} (0, x, \cdot) = \delta_{x}(\cdot).
\]
\end{theorem}
This theorem allows us to obtain different differential equations by making use of the symmetry of the BRW.

\begin{theorem}
\label{thm04} The moment $m_{1}(t, x, y)$ satisfies both integral equations
\begin{align*}
m_{1}(t, x, y) &= p(t, x, y) + \sum_{j=1}^{N} \beta_{j} \int_{0}^{t} p(t-s, x, x_{j}) m_{1}(t-s, x_{j}, y)\,ds,\\
m_{1}(t, x, y) &= p(t, x, y) + \sum_{j=1}^{N} \beta_{j} \int_{0}^{t} p(t-s, x_{j}, y) m_{1}(t-s, x, x_{j})\,ds.
\end{align*}
Similarly, the moment $m_{1}(t, x)$ satisfies both integral equations
\begin{align}
m_{1}(t, x) &= 1 + \sum_{j=1}^{N} \beta_{j} \int_{0}^{t} p(t-s, x, x_{j}) m_{1}(s, x_{j}) ds,\\
\label{eq015} m_{1}(t, x) &= 1 + \sum_{j=1}^{N} \beta_{j} \int_{0}^{t} m_{1}(s, x, x_{j}) ds.
\end{align}
The moments $m_{k}(t, x, y)$ and $m_{k}(t, x)$ for $k > 1$ satisfy the equations
\begin{align*}
m_{k}(t, x, y) &= m_{1}(t, x, y) +\notag\\&+ \sum_{j=1}^{N} \int_{0}^{t} m_{1}(t-s, x, x_{j}) g_{k}^{(j)} \left(m_{1}(s, x_{j}, y), \ldots, m_{k-1}(s, x_{j}, y) \right)\,ds,\\
m_{k}(t, x) &= m_{1}(t, x) +\notag\\&+ \sum_{j=1}^{N} \int_{0}^{t} m_{1}(t-s, x, x_{j}) g_{k}^{(j)} \left(m_{1}(s, x_{j}), \ldots, m_{k-1}(s, x_{j}) \right)\,ds .
\end{align*}
\end{theorem}
This theorem allows us to transition from differential equations to integral equations. It is later used to prove Theorem~\ref{thm06}.


\section{Properties of the operator $\mathscr{H}$.}\label{S:WSBRW}
We call a BRW supercritical if the local and global numbers of particles $\mu_{t}(y)$ and
$\mu_{t}$ grow exponentially. As was mentioned in the Introduction, one of the main results of this work is the equations~\eqref{eq1}, from which it follows that a BRW with several branching sources is supercritical if the operator $\mathscr{H}$ has a positive eigenvalue
 $\lambda$. For this reason we dedicate this section to a further examination of the spectral properties of the operator
$\mathscr{H}$.

We first mention an important statement proved
in~\cite[Lemma~3.1.1]{YarBRW:e}.

\begin{lemma}\label{lem01}
The spectrum $\sigma (\mathscr{A})$ of the operator $\mathscr{A}$ is included
in the half-line $(-\infty, 0]$. Also, since the operator $\sum_{j=1}^{N}
\beta_{j} \Delta_{x_{j}}$ is compact, $\sigma_{ess}(\mathscr{H}) = \sigma
\left(\mathscr{A} \right) \subset (-\infty, 0]$, where
$\sigma_{ess}(\mathscr{H})$ denotes the essential spectrum~\cite{Kato:e} of
the operator $\mathscr{H}$.
\end{lemma}

The following theorem provides a criterion of there being a positive
eigenvalue in the spectrum of the operator $\mathscr{H}$.

\begin{theorem}
\label{thm002} A number $\lambda > 0$ is an eigenvalue and  $f \in
l^{2}\left(\Z^{d} \right)$ is the corresponding eigenvector of the operator
$\mathscr{H}$ if and only if the system of linear equations
\begin{equation}\label{E:fxi}
f(x_{i}) = \sum_{j=1}^{N} \beta_{j}f(x_{j}) I_{x_{j} - x_{i}} (\lambda),\qquad i = 1, \ldots, N
\end{equation}
with respect to the variables $f(x_{i})$, where
\[
I_{x}(\lambda) := G_{\lambda}(x, 0) =
\frac{1}{(2\pi)^{d}} \int_{[-\pi, \pi]^{d}} \frac{e^{-i(\theta, x)}}{\lambda - \phi(\theta)} d\theta,\qquad x \in \Z^{d},
\]
has a non-trivial solution.
\end{theorem}

\begin{proof}
For $\lambda > 0$ to be an eigenvalue of the operator
$\mathscr{H}$ it is necessary and sufficient that there be a non-zero element $f
\in l^{2}\left(\Z^{d} \right) $ that satisfies the equation
\[
\left(\mathscr{H} - \lambda I \right)f =
\left(\mathscr{A} + \sum_{j=1}^{N} \beta_{j}\Delta_{x_{j}} - \lambda I \right)f = 0.
\]
Since $(\Delta_{x_{j}} f)(x) := f(x)\delta_{x_{j}}(x) =
f(x_{j})\delta_{x_{j}}(x)$, the preceding equality can be rewritten as follows:
\[
(\mathscr{A}f)(x) + \sum_{j=1}^{N} \beta_{j}f(x_{j})\delta_{x_{j}}(x) = \lambda f(x),\qquad x \in \Z^{d}.
\]
By applying the Fourier transform to this equality, we obtain
\begin{equation}\label{eq999}
(\widetilde{\mathscr{A}f})(\theta) + \sum_{j=1}^{N} \beta_{j} f(x_{j}) e^{i(\theta, x_{j})} = \lambda \tilde{f}(\theta),\qquad \theta \in [-\pi, \pi]^{d}.
\end{equation}
Here the Fourier transform $\widetilde{\mathscr{A}f}$ of the function
$(\mathscr{A}f)(x)$ is of the form $\phi \tilde{f}$, where $\tilde{f}$ is the
Fourier transform of the function $f$, and the function $\phi(\theta)$ is
defined by the equality~\eqref{E:Fourier}, see~\cite[Lemma 3.1.1]{YarBRW:e}. With
this in mind, we rewrite the equality~\eqref{eq999} as
\[
\phi(\theta) \tilde{f}(\theta) + \sum_{j=1}^{N} \beta_{j} f(x_{j}) e^{i(\theta, x_{j})} =
\lambda \tilde{f}(\theta),\qquad \theta \in [-\pi, \pi]^{d},
\]
or
\begin{equation}\label{eq998}
\tilde{f}(\theta) = \frac{1}{\lambda - \phi(\theta)} \sum_{j=1}^{N}  \beta_{j} f(x_{j}) e^{i(\theta, x_{j})},\qquad \theta \in [-\pi, \pi]^{d}.
\end{equation}
Since $\lambda > 0$ and $\phi(\theta) \leqslant 0$,  $
\int_{[-\pi, \pi]^{d}} |\lambda - \phi(\theta)|^{-2} \,d\theta < \infty$, which allows us to apply the inverse Fourier transform to~\eqref{eq998}:
\begin{equation}\label{eq997}
f(x)= \sum_{j=1}^{N} \beta_{j} f(x_{j}) I_{x_{j} - x}(\lambda),\qquad x \in \Z^{d}.
\end{equation}

Finally, we note that any solution of the system~\eqref{E:fxi}
completely defines the function $f(x)$ on the entirety of its domain by the formula~\eqref{eq997}, which proves the theorem.
\end{proof}

\begin{corollary}
\label{cor1} The number of positive eigenvalues of the
$\mathscr{H}$, counted with their multiplicity, does not exceed $N$.
\end{corollary}

\begin{proof}
Suppose the contrary is true. Then there are at least
$N+1$ linearly independent eigenvectors $f_{i}$ of
$\mathscr{H}$. Since, as it was established in the proof of Theorem~\ref{thm002}, the function $f(x)$ satisfies the equality~\eqref{eq997},
where $\beta_{j} > 0$ for all $j$, and $I_{x_{j} - x} > 0$ for all $j$ and $x$, the linear independence of the vectors $f_{i}$ is equivalent to the linear independence of the vectors
\[
\widehat{f}_{i} := \left(f_{i}(x_{1}), \ldots, f_{i}(x_{N}) \right),\qquad i = 1, \ldots, N+1.
\]
Given that such a set of $N+1$ vectors of dimension $N$ is always linearly dependent, so is the initial set of the vectors $f_{i}$, which contradicts our assumption.
\end{proof}

Let us introduce the matrix
\begin{equation}\label{E:defG}
G(\lambda) :=
\begin{pmatrix}
\beta_{1} I_{0}(\lambda)  & \beta_{2} I_{x_{2} - x_{1}}(\lambda) & \cdots & \beta_{N} I_{x_{N} - x_{1}}(\lambda)\\
\beta_{1} I_{x_{1} - x_{2}} (\lambda) & \beta_{2} I_{0} (\lambda)  & \cdots & \beta_{N} I_{x_{N} - x_{2}}(\lambda)\\
\cdots & \cdots & \ddots & \cdots\\
\beta_{1} I_{x_{1} - x_{N}}(\lambda) & \beta_{2} I_{x_{2} - x_{N}} (\lambda) & \cdots & \beta_{N} I_{0}(\lambda)
\end{pmatrix}.
\end{equation}

\begin{corollary}\label{corD} A number $\lambda > 0$ is an eigenvalue of
$\mathscr{H}$ if and only if $1$ is an eigenvalue of the matrix $G(\lambda)$, or, in other words, when the equality
\[
\det(G(\lambda)-I)=0
\]
holds.
\end{corollary}

\begin{proof}
This statement is a reformulation of the sufficient and necessary condition for consistency of the system~\eqref{E:fxi}.
\end{proof}

\begin{corollary}
\label{cor2} Let $\lambda_{0}>0$ be the largest eigenvalue of the operator
$\mathscr{H}$. Then $\lambda_{0}$ is a simple eigenvalue of $\mathscr{H}$, and $1$ is the largest eigenvalue of the matrix
$G(\lambda_{0})$.
\end{corollary}

\begin{proof}
Let us first demonstrate that if $\lambda_{0}$ is the largest eigenvalue of the operator $\mathscr{H}$, then $1$ is the largest (by absolute value) eigenvalue of the matrix $G(\lambda_{0})$. Indeed, assume it is not the case.

It follows from Corollary~\ref{corD} that $\lambda_{0}>0$ is an eigenvalue of
$\mathscr{H}$ if and only if $1$ is an eigenvalue of the matrix
$G(\lambda_{0})$. By the Perron-Frobenius theorem,
see~\cite[Theorem~8.4.4]{HJ:e}, which is applicable to the matrix
$G(\lambda_{0})$ since all its elements are strictly positive, the matrix
$G(\lambda_{0})$ has a strictly positive eigenvalue that is strictly greater
(by absolute value) than any other of its eigenvalues. We denote this
dominant eigenvalue by $\gamma(\lambda_{0})$. Then $\gamma(\lambda_{0})> 1$,
since we assumed that $1$ is not the largest eigenvalue of $G(\lambda_{0})$.
Given that the functions $I_{x_{i} - x_{j}}(\lambda)$ are continuous with
respect to $\lambda$, all elements of $G(\lambda)$, and therefore all
eigenvalues of $G(\lambda)$ are continuous functions of $\lambda$. Because
for all $i$ and $j$ $I_{x_{i} - x_{j}}(\lambda) \to 0$ as $\lambda \to
\infty$, all eigenvalues of the matrix $G(\lambda)$ tend to zero as $\lambda
\to \infty$. Therefore there is such a $\hat{\lambda} > \lambda_{0}$ that
$\gamma(\hat{\lambda}) = 1$. Corollary~\ref{corD} states that this
$\hat{\lambda}$ then has to be an eigenvalue of the operator $\mathscr{H}$,
which contradicts our initial assumption that  $\lambda_{0}$ is the largest
eigenvalue of $\mathscr{H}$.

We have just proved that $1$ is the largest eigenvalue of the matrix $G(\lambda_{0})$; it then follows from the Perron-Frobenius theorem that this eigenvalue is simple. Now, in order to complete the proof we only have to show that the eigenvalue
$\lambda_{0}$ of the operator $\mathscr{H}$ is also simple.

Assume it is not the case, and $\lambda_{0}$ is not simple. Then there are at least two linearly independent eigenvectors $f_{1}$ and $f_{2}$ corresponding to the eigenvalue
$\lambda_{0}$. We then can, by applying the equality~\eqref{eq997} once again, see that the linear independence of the vectors $f_{1}$ and $f_{2}$ is equivalent to the linear independence of the vectors
\[
\hat{f}_{i} := \left(f_{i}(x_{1}), \ldots, f_{i}(x_{N}) \right),\qquad i = 1, 2.
\]
It also follows from Theorem~\ref{thm002} and the definition of
$G(\lambda)$ that both vectors $\hat{f}_{i}$ satisfy the system of linear equations $\left(G(\lambda_{0}) - I\right) f = 0$, which contradicts the simplicity of eigenvalue $1$ of $G(\lambda_{0})$.
This completes the proof.
\end{proof}

We will also need the following result \cite[Corollary~8.1.29]{HJ:e}.
\begin{lemma}
\label{lem001} Let the elements of a matrix $G$ and vector $f$ be strictly positive. Let us also assume that
$\left(G f \right)_{i} > f_{i}$ for all $i=1,\ldots,N$. Then the matrix
$G$ has an eigenvalue $\gamma > 1$.
\end{lemma}

\begin{corollary}
\label{cor3} The largest eigenvalue $\gamma(\lambda)$ of the matrix
$G(\lambda)$ is a continuous strictly decreasing function for $\lambda > 0$.
\end{corollary}

\begin{proof}
The continuity of $\gamma(\lambda)$ follows from the fact that the elements of $G(\lambda)$ are themselves continuous functions of $\lambda$. We now prove the decreasing monotonicity of
$\gamma(\lambda)$. Assume the contrary: let there be such two numbers $\lambda' >
\lambda'' > 0$ that $\gamma(\lambda') \geqslant \gamma(\lambda'') > 0$.
We denote by $f$ an eigenvector of $G(\lambda')$ corresponding to the eigenvalue $\gamma(\lambda')$. By the Perron-Frobenius theorem this vector can be chosen uniquely up to multiplication by a constant, and can furthermore be chosen to be strictly positive. Let us set
\[
G'' := \frac{1}{\gamma(\lambda')} G(\lambda''),\qquad
G' := \frac{1}{\gamma(\lambda')} G(\lambda').
\]
Then $G' f = f$, and the largest eigenvalue of the matrix $G''$ does not exceed
$1$. Also, since all elements of the matrices $G'$ and $G''$ are monotonously decreasing strictly positive functions of $\lambda$, $\left(G''
f\right)_{i} > f_{i}$ for $i=1,\ldots,N$, which contradicts Lemma~\ref{lem001} and concludes the proof.
\end{proof}

\begin{corollary}
\label{cor4} Let the operator $\mathscr{H}$ have an eigenvalue
$\lambda
> 0$. Consider the operator
$\mathscr{H}' = \mathscr{A} + \sum_{j=1}^{N} \beta'_{j} \Delta_{x_{j}}$ with parameters
$\beta_{i}'$, $i=1,\ldots,N$ that satisfy the inequalities
$\beta_{j}' \geqslant \beta_{j}$ for $j=1,\ldots,N$. Moreover, let there be such an $i$ that $\beta_{i}' > \beta_{i}$. Then the operator
$\mathscr{H}'$ has an eigenvalue $\lambda' > \lambda$.
\end{corollary}

\begin{proof}
It suffices to show that the matrix $G'(\lambda)$ corresponding to the operator
$\mathscr{H}'$ and defined according to~\eqref{E:defG}
has an eigenvalue $1$ for some $\lambda'>\lambda$. Let us first demonstrate that the matrix $G'(\lambda)$ has an eigenvalue
$\gamma'> 1$.

Since we assumed $\lambda$ is an eigenvalue of the operator
$\mathscr{H}$, it follows from Corollary~\ref{corD} that $1$ is an eigenvalue of the matrix $G(\lambda)$. Now, as all elements of the matrix
$G(\lambda)$ are strictly positive, by applying the Perron-Frobenius theorem we conclude that
$G(\lambda)$ has the strictly largest (by absolute value) eigenvalue
$\gamma\ge1$ with a corresponding strictly positive eigenvector $f$. Therefore,
\begin{equation}\label{EGl0}
\left(G(\lambda) f\right)_{i} = \gamma f_{i}\geqslant f_{i},\qquad i=1,\ldots,N,
\end{equation}
By assumption, the following inequalities hold:
\[
\beta'_{j} I_{x_{i} - x_{j}}(\lambda)
\geqslant \beta_{j} I_{x_{i} - x_{j}}(\lambda) > 0,\qquad i,j=1,\ldots,N;
\]
moreover,
\[
\beta'_{i} I_{x_{i} - x_{j}}(\lambda) > \beta_{i} I_{x_{i} - x_{j}}(\lambda) > 0.
\]
It then follows from~\eqref{EGl0} that
\[
\left(G'(\lambda) f\right)_{i} > \gamma f_{i}\geqslant f_{i},\qquad i=1,\ldots,N.
\]
We now obtain from Lemma~\ref{lem001} that the matrix $G'(\lambda)$ has an eigenvalue $\gamma'> 1$. Since its largest eigenvalue $\gamma(\lambda)$ is a continuous function of $\lambda$ that tends to zero as $\lambda\to\infty$, there is such a
$\lambda'>\lambda$ that $\gamma(\lambda') = 1$. This completes the proof.
\end{proof}

\begin{corollary}\label{cor5}
Let the operator $\mathscr{H}$ have the largest eigenvalue $\lambda_{0}
> 0$. Consider the operator $\mathscr{H}' = \mathscr{A} + \sum_{j=1}^{N}
\beta'_{j} \Delta_{x_{j}}$ with parameters $\beta_{i}'$, $i=1,\ldots,N$
that satisfy the inequalities $\beta_{j}' \leqslant \beta_{j}$ for
$j=1,\ldots,N$. Moreover, let there be such an $i$ that $\beta_{i}' <
\beta_{i}$. Then all eigenvalues of the operator $\mathscr{H}'$ are strictly less (by absolute value) than $\lambda_{0}$.
\end{corollary}

\begin{proof}
This statement immediately follows from the corollary above.
\end{proof}

%

\begin{lemma}\label{lem03} Let $\mathscr{H}$ be a continuous self-adjoint operator on a separable Hilbert space $E$, the spectrum of which is a disjoint union of two sets: a finite (counting multiplicity) set of isolated eigenvalues $\lambda_{i} > 0$ and the remaining part of the spectrum which is included in $[-s, 0]$, $s> 0$. Then the solution  $\,m(t)$ of the Cauchy problem
\begin{equation}
\label{eq030}
\frac{d m(t)}{dt} = \mathscr{H} m(t),\qquad m(0) = m_{0},
\end{equation}
satisfies the condition
\[
\lim_{t \to \infty} e^{-\lambda_{0} t} m(t) = C\left(m_{0}\right),
\]
where $\lambda_{0} = \max_{i} \lambda_{i}$.
\end{lemma}

\begin{proof}
We denote by $V_{\lambda_{i}}$ the finite-dimensional eigenspace of
$\mathscr{H}$ corresponding to the eigenvalue $\lambda_{i}$.
%
Consider the projection $P_{i}$ of $\mathscr{H}$ onto $V_{\lambda_{i}}$,
see~\cite{Kato:e}. Let
\begin{align*}
x_{i}(t) &:= P_{i} m(t),\\
v(t) &:= \left(I - \sum_{i} P_{i} \right) m(t) = m(t) - \sum_{i} x_{i}(t).
\end{align*}
It is known, see~\cite{Kato:e}, that all spectral operators $P_{i}$ and
$\left(I - \sum P_{i} \right)$ commute with $\mathscr{H}$. Therefore
\begin{align*}
\frac{d x_{i}(t)}{dt} &= P_{i} \mathscr{H} m(t) = \mathscr{H} x_{i}(t)\\
\frac{d v(t)}{dt} &= \left(I - \sum P_{i} \right) \mathscr{H} m(t) = \left(I - \sum P_{i} \right)  \mathscr{H} \left(I - \sum P_{i} \right) v(t).
\end{align*}

As $x_{i}(t) \in V_{\lambda_{i}}$, we can see that $\mathscr{H} x_{i}(t) =
\lambda_{i} x_{i}(t)$, from which it follows that $x_{i}(t) = e^{\lambda_{i}
t} x_{i}(0)$. Also, since the spectrum of the operator $\mathscr{H}_{0} :=
\left(I - \sum P_{i} \right) \mathscr{H} \left(I - \sum P_{i} \right)$ is
included into the spectrum of $\mathscr{H}$ and does not contain any of the
isolated eigenvalues $\lambda_{i}$, it is included into $[-s, 0]$. From this
we obtain $|v(t)| \leqslant |v(0)|$ for all $t\geqslant0$,
see~\cite[Lemma~3.3.5]{YarBRW:e}. Therefore
\begin{equation}\label{E:mt}
m(t) = \sum_{i} e^{\lambda_{i} t} P_{i} m(0 ) + v(t),
\end{equation}
and the proof is complete.
\end{proof}

\begin{remark}\label{RemL}
Let $\lambda_{0}$ be the largest eigenvalue of the operator
$\mathscr{H}$. Then due to~\eqref{E:mt}
$C(m_{0})=P_{0} m(0 )$. Therefore $C(m_0)\neq 0$ if and only if the orthogonal projection  $P_{0} m(0 )$ of the initial value $m_0=m(0)$ onto the eigenspace corresponding to the eigenvalue $\lambda_{0}$ is non-zero.

If the eigenvalue $\lambda_{0}$ of the operator $\mathscr{H}$
is simple and $f$ is a corresponding eigenvector, the projection $P_{0}$
is defined by the formula $P_{0}x=\frac{(f,x)}{(f,f)}f$, where $(\cdot,\cdot)$
is the scalar product in the Hilbert space $E$.
In cases when this $\lambda_{0}$ is not simple, describing the projection $P_{0}$ is a significantly more difficult task.

We remind the reader that we proved the simplicity of the largest eigenvalue of $H$ above, which allows us to bypass this complication.
\end{remark}

\begin{theorem}
\label{thm06} Let the operator $\mathscr{H}$, defined as in~\eqref{E:defHbbb} with parameters $\lbrace
\beta_{i} \rbrace_{i=1}^{N}$, have a finite (counting multiplicity) number of positive eigenvalues. We denote the largest of them by $\lambda_{0}$, and the corresponding normalized vector by $f$. Then for all $n
\in\N$ and $t \to \infty$ the following limit statements hold:
\begin{equation}\label{E:mainmom}
m_{n}(t, x, y) \sim C_{n}(x, y) e^{n \lambda_{0} t},\quad
m_{n}(t, x) \sim C_{n}(x) e^{n\lambda_{0} t},
\end{equation}
where
\[
C_{1} (x, y) =  f(y) f(x),\qquad
C_{1}(x) =  f(x)\frac{1}{\lambda_{0}} \sum_{j=1}^{N} \beta_{j} f(x_{j}),
\]
and for $n \geqslant 2$ the functions $C_{n}(x, y)$ and $C_{n}(x) > 0$ are defined by the equalities below:
\begin{align*}
C_{n}(x, y) &= \sum_{j=1}^{N} g^{(j)}_{n}\left(C_{1}(x_{j}, y), \ldots, C_{n-1}(x_{j}, y) \right) D^{(j)}_{n}(x),\\
C_{n}(x) &= \sum_{j=1}^{N} g^{(j)}_{n}\left(C_{1}(x_{j}), \ldots, C_{n-1}(x_{j}) \right) D^{(j)}_{n}(x),
\end{align*}
where $g^{(j)}_{n}$ are the functions defined in~\eqref{E:defgnj}
and $D^{(j)}_{n}(x)$ are certain functions that satisfy the estimate $|D_{n}^{(j)} (x)|\leqslant
\frac{2}{n\lambda_{0}}$ for $n\geqslant
n_{*}$ and some $n_{*}\in\N$.
\end{theorem}

\begin{proof}
For $n \in\N$ we introduce the functions
\[
\nu_{n} := m_{n} (t, x, y) e^{-n \lambda_{0} t}.
\]
We obtain from Theorem~\ref{thm02} (see equations~\eqref{eq007} and~\eqref{eq008} for
$m_{n}$) the following equations for $\nu_{n}$:
\begin{align*}
\frac{d \nu_{1}}{dt} &= \mathscr{H} \nu_{1} - \lambda_{0} \nu_{1},\\
\frac{d \nu_{n}}{dt} &= \mathscr{H} \nu_{n} - n \lambda_{0} \nu_{n} +
\sum_{j=1}^{N} \Delta_{x_{j}} g_{n}^{(j)} \left(\nu_{1}, \ldots, \nu_{n-1} \right),\qquad n \geqslant 2,
\end{align*}
the initial values being $\nu_{n}(0, \cdot, y) = \delta_{y}(\cdot), n \in\N$.

Since $\lambda_{0}$ is the largest eigenvalue of $\mathscr{H}$, for $n
\geqslant 2$ the spectrum of the operator $\mathscr{H}_{n} := \mathscr{H} - n
\lambda_{0} I$ is included into $(-\infty, -(n-1)\lambda_{0}]$. As it was
shown, for example, in \cite[p.~58]{YarBRW:e}, if the spectrum of a
continuous self-adjoint operator $\widetilde{\mathscr{H}}$ on a Hilbert space
is included into $(-\infty, -s], s > 0$, and also $f(t) \to f_{*}$ as $t \to
\infty$, then the solution of the differential equation
\[
\frac{d \nu}{dt} = \widetilde{\mathscr{H}} \nu + f(t)
\]
satisfies the condition $\nu(t) \to -\widetilde{\mathscr{H}}^{-1} f_{*}$.
For this reason for $n \geqslant 2$ we obtain
\begin{multline*}
C_{n}(x, y) = \lim_{t \to \infty} \nu_{n} =
-\sum_{j=1}^{N} \left(\mathscr{H}_{n}^{-1} \Delta_{x_{j}}g_{n}^{(j)}(C_{1}(\cdot, y), \ldots, C_{n-1}(\cdot, y) )\right)(x)=\\
-\sum_{j=1}^{N} g_{n}^{(j)}(C_{1}(x_{j}, y), \ldots, C_{n-1}(x_{j}, y) )(\mathscr{H}_{n}^{-1} \delta_{x_{j}}(\cdot))(x)).
\end{multline*}

Let us now prove the existence of such a natural number $n_{*}$ that for all
$n\geqslant n_{*}$ the estimates
\[
D_{n}^{(j)} (x) := |(\mathscr{H}_{n}^{-1} \delta_{x_{j}}(\cdot))(x)|
\leqslant \frac{2}{n\lambda_{0}}
\]
hold.
We first evaluate the norm of the operator $\mathscr{H}_{n}^{-1}$. For this purpose, let us introduce two vectors $x$ and $u$ such that $u=\mathscr{H}_{n}x=
\mathscr{H}x - n\lambda_{0} x$. Then $\|u\|\geqslant  n\lambda_{0} \|x\| -
\|\mathscr{H}x\|\geqslant (n\lambda_{0} -\|\mathscr{H}\|)\|x\|$, hence
$\|\mathscr{H}_{n}^{-1}u\|=\|x\|\leqslant \|u\|/\left(n\lambda_{0}
-\|\mathscr{H}\|\right)$, and therefore for all $n\geqslant
n_{*}=2\lambda_{0}^{-1}\|\mathscr{H}\|$ the estimate
\[
\|\mathscr{H}_{n}^{-1}\|\leqslant \frac{2}{n\lambda_{0}}
\]
holds. From this we conclude that
\[
|(\mathscr{H}_{n}^{-1} \delta_{x_{j}}(\cdot))(x)| \leqslant
\|\mathscr{H}_{n}^{-1} \delta_{x_{j}}(\cdot)\|\leqslant
\|\mathscr{H}_{n}^{-1} \| \|\delta_{x_{j}}(\cdot) \| \leqslant \frac{2}{n\lambda_{0}},\qquad n\geqslant n_{*}.
\]

Let us now turn to estimating the asymptotic behaviour of particle number
moments. It follows from~\eqref{eq015} that as $t \to \infty$ the following
asymptotic equivalences hold:
\begin{equation}\label{eq1298}
m_{1}(t, x) \sim \sum_{j=1}^{N} \beta_{j} \int_{0}^{t} m_{1}(s, x, x_{j}) \,ds
\sim \sum_{j=1}^{N} \frac{\beta_{j}}{\lambda_{0}} m_{1}(t, x, x_{j}).
\end{equation}
Since the function $m_{1}(t, x, 0)$ exhibits exponential growth as $t \to \infty$, the function $m_{1}(t, x)$
will display the same behaviour.

We can now infer the asymptotic behaviour of the higher moments $m_{n}(t, x)$ for $n \geqslant 2$ from the equations~\eqref{eq008} in much the same way it was done above for the higher moments $m_{n}(t, x, y)$.

We now proceed to prove the equalities for $C_{1} (x, y)$ and $ C_{1}(x)$. By Corollary~\ref{cor2} the eigenvalue $\lambda_{0}$ is simple, from which it follows, according to Remark~\ref{RemL}, that
\[
C_{1} (x, y) = \lim_{t \to \infty} e^{-\lambda_{0} t} m_{1}(t, x, y) = Pm_{0} =
\left( m_{1}(0, x, y), f\right) f(x).
\]
But $m_{1}(0, x, y) = \delta_{y}(x)$, therefore
\[
C_{1} (x, y) = \left( m_{1}(0, x, y), f\right) f(x) = f(y) f(x).
\]
We also obtain from~\eqref{eq1298} that
\[
C_{1}(x) = \frac{1}{\lambda_{0}} \sum_{j=1}^{N} \beta_{j} C_{1}(x, x_{j}) =  f(x)\frac{1}{\lambda_{0}} \sum_{j=1}^{N} \beta_{j} f(x_{j}),
\]
which concludes the proof.\end{proof}

\begin{corollary}
\label{cor6} $C_{n}(x, y) = \psi^{n}(y) C_{n}(x)$, where $ \psi(y) =
\frac{\lambda_{0}f(y)}{\sum_{j=1}^{N} \beta_{j} f(x_{j})}$.
\end{corollary}

\begin{proof}
We prove the corollary by induction on $n$. The induction basis for $n=1$ holds due to Theorem~\ref{thm06}. Let us now deal with the induction step: according to Theorem~\ref{thm06},
\begin{align}
\label{eq1998}
C_{n+1}(x, y) &= \sum_{j=1}^{N} g^{(j)}_{n+1}\left(C_{1}(x_{j}, y), \ldots, C_{n}(x_{j}, y) \right) D^{(j)}_{n+1}(x),\\
\label{eq1999}
C_{n+1}(x) &= \sum_{j=1}^{N} g^{(j)}_{n+1}\left(C_{1}(x_{j}), \ldots, C_{n}(x_{j}) \right) D^{(j)}_{n+1}(x);
\end{align}
so it suffices to prove that for all $j$ the equalities
\[
g^{(j)}_{n+1}\left(C_{1}(x_{j}, y), \ldots, C_{n}(x_{j}, y) \right) = \psi^{n+1}(y) g^{(j)}_{n+1}\left(C_{1}(x_{j}), \ldots, C_{n}(x_{j}) \right)
\]
hold.
As it follows from the definition and the induction hypothesis,
\begin{multline*}
g^{(j)}_{n+1}\left(C_{1}(x_{j}, y), \ldots, C_{n}(x_{j}, y) \right) =\\=
\sum_{r=2}^{n+1} \frac{\beta_{j}^{(r)}}{r!} \sum_{\substack{i_{1}, \ldots, i_{r} > 0 \\
i_{1} + \cdots + i_{r} = n+1}} \frac{n!}{i_{1}! \cdots i_{r}!} C_{i_{1}}(x_{j}, y) \cdots C_{i_{r}}(x_{j}, y) =\\
= \psi^{n+1} (y) \sum_{r=2}^{n+1} \frac{\beta_{j}^{(r)}}{r!} \sum_{\substack{i_{1}, \ldots, i_{r} > 0 \\ i_{1} + \cdots + i_{r} = n+1}} \frac{n!}{i_{1}! \cdots i_{r}!} C_{i_{1}}(x_{j}) \cdots C_{i_{r}}(x_{j}),
\end{multline*}
which proves the corollary.
\end{proof}

\section{Proof of Theorem~\ref{thm1}.}\label{S:limthm}
We will need a few auxiliary lemmas. Let us introduce the function
\begin{equation}\label{def1}
f(n, r) := \sum_{\substack{i_{1}, \ldots, i_{r} > 0 \\ i_{1} + \cdots + i_{r} = n}} i_{1}^{i_{1}}  \cdots  \,i_{r}^{i_{r}},\qquad 1 \leqslant r \leqslant n.
\end{equation}

\begin{lemma}
\label{lem2} $f(n, n) = 1$ and $f(n, 1) = n^{n}$; for
$2 \leqslant r \leqslant n$ the following formula holds:
\[
f(n, r) = \sum_{u=1}^{n-r+1} u^{u} f(n-u, r-1).
\]
\end{lemma}

\begin{proof}
To prove the lemma, group all addends in~\eqref{def1} by the possible values of
$i_{1}$; since $1 \leqslant i_{1} \leqslant n-r+1$,
\[
f(n, r) = \sum_{i_{1} = 1}^{ n-r+1} i_{1}^{i_{1}} \sum_{\substack{i_{2}, \ldots, i_{r} > 0 \\
i_{2} + \cdots + i_{r} = n-i_{1}}} i_{2}^{i_{2}}  \cdots i_{r}^{i_{r}} = \sum_{u=1}^{n-r+1} u^{u} f(n-u, r-1),
\]
which completes the proof.
\end{proof}

\begin{lemma}
\label{lem3} The function $g(x) = x^{x} (n-x)^{n-x}$, $x \in [1, n-1]$, attains its maximum at the ends of its domain.
\end{lemma}

\begin{proof}
By applying the logarithm to both sides of the equation above, we obtain
\[
\ln g(x) = x\ln x + (n-x) \ln (n-x),
\]
from which it follows that
\[
\left( \ln g(x) \right)' = \ln x + 1 - \ln (n-x) - 1 = \ln x - \ln (n-x).
\]
This means that $\left( \ln g(x) \right)' < 0$ for $x < \frac{n}{2}$, and $\left( \ln
g(x) \right)' > 0$ for $x > \frac{n}{2}$. Therefore, the function $g(x)$
is decreasing when $x < \frac{n}{2}$ and increasing when $x > \frac{n}{2}$, which concludes the proof.
\end{proof}

\begin{remark}
The lemma above holds for any other intercept included in the original domain $[1,
n-1]$ (with the only difference being that the function's values at the ends of the new intercept can be different.)
\end{remark}

\begin{lemma}
\label{lem4} For all $n \geqslant 2$ the inequality $f(n, 2) < 6
(n-1)^{n-1}$ holds.
\end{lemma}

\begin{proof}
We prove the lemma by induction on $n$. The induction basis for $n=2$ and $n>3$ holds: $f(2, 2) = 1 < 6$, $f(3, 2) =
4 < 24$.

Let us now turn to the induction step: to complete the proof, we have to show that the statement of the lemma holds for any $n>3$ if it holds for all the preceding values of $n$. By applying Lemmas~\ref{lem2} and~\ref{lem3} and evaluating the sum by the maximum term multiplied by their number, we obtain
\begin{multline*}
f(n, 2) = \sum_{u=1}^{n-2+1} u^{u} f(n - u, 2- 1) =\\=
\sum_{u=1}^{n-1} u^{u} (n-u)^{n-u} =  2 (n-1)^{n-1} +
\sum_{u=2}^{n-2} u^{u} (n-u)^{n-u} \leqslant\\
\leqslant 2(n-1)^{n-1} + 4 (n-3) (n-2)^{n-2} < 6 (n-1)^{n-1},
\end{multline*}
which concludes the proof.\end{proof}

\begin{lemma}
\label{lem5} For all $n \geqslant 3$ the inequality $f(n, 3) < 6
(n-1)^{n-1}$ holds.
\end{lemma}

\begin{proof}
We prove the lemma by induction on $n$. The induction basis for $n>3$ holds: $f(3, 3) = 1 < 24$.

As for the inductive step, according to Lemma~\ref{lem4}
\begin{multline*}
f(n, 3) = \sum_{u=1}^{n-2} u^{u} f(n-u, 2) \leqslant
6 \sum_{u=1}^{n-2}  u^{u} (n-u-1)^{n-u-1}\leqslant\\ \leqslant 6 (n-2)^{n-2} (n-2) = 6 (n-2)^{n-1} < 6 (n-1)^{n-1},
\end{multline*}
which completes the proof.
\end{proof}

\begin{lemma}
\label{lem6} For all $n \geqslant r$ and $r \geqslant 2$ the inequality
$f(n, r) < 6 (n-1)^{n-1}$ holds.
\end{lemma}

\begin{proof}
We prove the lemma by induction on $n$. The induction basis and the cases $r = 2$ and $r = 3$ were considered in Lemmas~\ref{lem4} and~\ref{lem5}.

We can therefore assume that $r \geqslant 4$. Let us now prove the inductive step. By Lemma~\ref{lem2}
\begin{equation}
\label{eq2}
f(n, r) = \sum_{u=1}^{n-r+1} u^{u} f(n-u, r-1) \leqslant 6 \sum_{u=1}^{n-r+1} u^{u} (n-u-1)^{n-u-1}.
\end{equation}
It follows from Lemma~\ref{lem3} that the function $g(u) := u^{u} (n-u-1)^{n-u-1}$ attains its maximum value at (one of) the ends of the intercept $[1, n-r+1]$. These values are
\[
g(1) = (n-2)^{n-2},\qquad
g(n-r+1) = (n-r+1)^{n-r+1} (r-2)^{r-2}\,.
\]
Consider $g(n-r+1)$. By applying Lemma~\ref{lem3} once again, we obtain $h(r) := g(n-r+1)$
attains its maximum value at (one of) the ends of the intercept $[3, n]$, these values being
\[
h(3) = (n-2)^{n-2},\qquad h(n) = (n-2)^{n-2}.
\]
So the largest term in the right-hand side of~\eqref{eq2} is
$g(1) =(n-2)^{n-2}$. Therefore,
\[
6 \sum_{u=1}^{n-r+1} u^{u} (n-u-1)^{n-u-1} \leqslant 6 (n - r + 1) (n-2)^{n-2} < 6 (n-1)^{n-1},
\]
which concludes the proof.
\end{proof}

\begin{lemma}
\label{lem7} There is such a constant $C>0$ that for all $n \geqslant r
\geqslant 2$ the inequality
\begin{align}\label{eq019}
f(n, r) < C \frac{n^{n}}{r^{ r -1}}
\end{align}
holds.
\end{lemma}

\begin{proof}
We first introduce the quantities
\begin{align}\label{E:defn1}
n_{1}&=\max\{n\in\N:~ 6^{6}(n-5)^{n-5} \geqslant 4^{4} (n-3)^{n-3}\},\\
\label{E:defn2}
n_{2}&=\max\{n\in\N:~ 283 (n-2)^{n-2} \geqslant (n-1)^{n-1}\},\\
\label{E:defn3}
n_{3}&=\max\left\{n\in\N:~ \left(1+ \frac{2}{n-1} \right)^{n-1} \leqslant 2 e\right\},\\
\label{E:defntilde}
 \tilde{n} &= \max \{n_{1}, n_{2}, n_{3}\}
\end{align}
and then the following sets of ordered pairs $(n,r)$:
\begin{align*}
 \mathbb{D}&=\{(n,r):~ 2\leqslant r\leqslant n\},\\
\mathbb{D}_{1}&=\{(n,r):~ 2\leqslant r\leqslant n\leqslant \tilde{n}\},\\
\mathbb{D}_{2}&=\{(n,r):~ 2\leqslant r\leqslant 6,~
r\leqslant n\},\\
\mathbb{D}_{3}&=\{(n,r):~ r = n \text{ or } r = n-1\},\\
\widetilde{\mathbb{D}} &= \mathbb{D}_{1} \cup \mathbb{D}_{2} \cup \mathbb{D}_{3}.
\end{align*}

It can easily be seen that $n_{1},n_{2},n_{3}<\infty$; in fact, $\tilde{n} =
106$. Therefore, the set $\mathbb{D}_{1}$ contains a finite number of pairs,
and we can pick a large enough $C$ for~\eqref{eq019} to hold for all pairs
from this set.

As follows from Lemma~\ref{lem6}, the same can be said of the set $\mathbb{D}_{2}$. Indeed, since for all $n \geqslant r
\geqslant 2$ the inequality $f(n, r) < 6(n-1)^{n-1}$ is true, we obtain
\begin{align*}
f(n, 2) &< 6 (n-1)^{n-1} = 6 \cdot 2^{2-1}  \frac{(n-1)^{n-1}}{2^{2-1}};\\
f(n, 3) &< 6 (n-1)^{n-1} = 6 \cdot 3^{3-2} \frac{(n-1)^{n-1}}{2^{2-1}};\\
&\cdots \\
f(n, 6) &< 6 (n-1)^{n-1} = 6 \cdot 6^{6-1} \frac{(n-1)^{n-1}}{6^{6-1}}.
\end{align*}
Therefore, for $C \geqslant 6^{6}$~\eqref{eq019} holds for any pair from the
set $\mathbb{D}_{2}$.

Finally,~\eqref{eq019} also holds for any pair from $\mathbb{D}_{3}$ with $C
\geqslant 1/2$, because by definition the following inequalities are true:
\[
f(n, n) = 1 < \frac{C n^{n}}{r^{r-1}} = Cn,\qquad
f(n, n-1) = 2(n-1) < \frac{C n^{n}}{r^{r-1}} = \frac{C n^{n}}{(n-1)^{n-2}}.
\]

From this it follows that the constant $C$ can be chosen large enough for
\eqref{eq019} to hold for any element of the set $\widetilde{\mathbb{D}}$. In addition, we set it to be large enough for the inequality to hold for all ordered pairs
$(\tilde{n}+1, r) \in \mathbb{D} \setminus \widetilde{\mathbb{D}}$, which can be done due to the number of these pairs being finite. We now fix $C$ according to these considerations.

Consequently, to complete the proof we only have to demonstrate that for the $C$ chosen above~\eqref{eq019} holds for all $(n, r) \in \mathbb{D} \setminus
\widetilde{\mathbb{D}}$. We do this by induction on
$n$, proving the statement on every step for all $r$ such that $(n, r) \in \mathbb{D} \setminus
\widetilde{\mathbb{D}}$.

By the definition of $\tilde{n}$ (see~\eqref{E:defntilde}) and thanks to the choice of $C$
we can use $n=\tilde{n}+1$ as induction basis. We now turn to the induction step: we assume that for some $n\geqslant \tilde{n}+1$ the statement of the lemma holds for all ordered pairs $(n,r) \in \mathbb{D} \setminus
\widetilde{\mathbb{D}}$ and prove that it then holds for all pairs $(n+1,
r) \in \mathbb{D} \setminus \widetilde{\mathbb{D}}$.

By Lemma~\ref{lem2} and the induction hypothesis
\[
f(n+1, r)= \sum_{u=1}^{n-r+2} u^{u} f(n+1-u, r-1) < \frac{C}{(r-1)^{ r-2}} \sum_{u=1}^{n-r+2} u^{u} (n+1-u)^{n-u+1}.
\]
We note that
\begin{equation}\label{E:sum1}
\sum_{u=1}^{n-r+2} u^{u} (n+1-u)^{n-u+1} = n^{n} + 4(n-1)^{n-1} + 27(n-2)^{n-2} + \sum_{u=4}^{n-r+2} u^{u} (n+1-u)^{n-u+1}.
\end{equation}
In order to evaluate the sum in the right-hand side of the equation, we point out that by Lemma~\ref{lem3} the terms $k(u) :=
u^{u} (n+1-u)^{n-u+1}$ in the sum attain their maximum values at the ends of $[4, n-r+2]$; these values are
\[
k(4) = 4^{4} (n-3)^{n-3},\qquad
k(n-r+2) = (n-r+2)^{n-r+2} (r-1)^{r-1}.
\]
To find out which one of these two values $k(4)$ and $k(n-r+2)$ is greater,
consider the function
\[
l(r) :=
(n-r+2)^{n-r+2}\cdot (r-1)^{r-1}.
\]
Since $(n,r)\not\in\widetilde{\mathbb{D}}$, $(n,r)\not\in\mathbb{D}_{2}
\cup \mathbb{D}_{3}$. Therefore, $r \in [7, n-2]$, and the function $l(r)$ assumes its maximum value at the ends of $[7, n-2]$; these values are
\[
l(7) = 6^{6} (n-5)^{n-5},\qquad l(n-2) = 4^{4} (n-3)^{n-3},
\]
and we have to find out, once again, which one of them is larger.

Again, since $(n,r)\not\in\widetilde{\mathbb{D}}$,
$(n,r)\not\in\mathbb{D}_{1}$, and $6^{6}(n-5)^{n-5} < 4^{4}
(n-3)^{n-3}$. Therefore, $l(7) \leqslant l(n-2)$, from which we obtain
\begin{equation}\label{E:estku}
k(u) := u^{u} (n+1-u)^{n-u+1}\leqslant 4^{4} (n-3)^{n-3}\quad\text{for}\quad n\in [4, n-r+2].
\end{equation}

We can now finally evaluate the sum in the right-hand side of~\eqref{E:sum1}. Since none of the terms $u^{u} (n+1-u)^{n-u+1}$ for $u \in [4, n-r+2]$ in the right-hand side of~\eqref{E:sum1} exceed $4^{4} (n-3)^{n-3}$, and the number of these terms does not exceed $n-8$,
\begin{equation}\label{E:sum2}
\sum_{u=4}^{n-r+2} u^{u} (n+1-u)^{n-u} \leqslant 4^{4} (n-8) (n-3)^{n-3} \leqslant 4^{4} (n-2)^{n-2}.
\end{equation}

We therefore obtain from~\eqref{E:sum1} and~\eqref{E:sum2} that
\begin{multline*}
\sum_{u=1}^{n-r+2} u^{u} (n+1-u)^{n-u+1}
\leqslant n^{n} + 4(n-1)^{n-1} + 27(n-2)^{n-2} + 4^{4} (n-2)^{n-2} =\\
= n^{n} + 4(n-1)^{n-1} + 283 (n-2)^{n-2}.
\end{multline*}

Since $(n,r)\not\in\widetilde{\mathbb{D}}$,
$(n,r)\not\in\mathbb{D}_{1}$, $283 (n-2)^{n-2} < (n-1)^{n-1}$, which allows us to rewrite the previous inequality as follows:
\[
\sum_{u=1}^{n-r+2} u^{u} (n+1-u)^{n-u+1}
\leqslant n^{n} + 4(n-1)^{n-1} + 283 (n-2)^{n-2} \leqslant n^{n} + 5(n-1)^{n-1}.
\]

Consequently,
\begin{multline}\label{E:sum233}
f(n+1, r)  \leqslant \frac{C}{(r-1)^{ r-2}} \left[n^{n} + 5(n-1)^{n-1} \right] =\\ = \frac{C (n+1)^{n+1}}{r^{r -1}} \frac{r^{r - 1}}{(r-1)^{r-2}}\left[\frac{n^{n}}{(n+1)^{n+1}} + \frac{5 (n-1)^{n-1}}{(n+1)^{n+1}}\right].
\end{multline}
It is obvious that
\[
\frac{r^{r - 1}}{(r-1)^{r-2}} \cdot\frac{n^{n}}{(n+1)^{n+1}} = \frac{\left(1 + \frac{1}{r-1} \right)^{r-1}}{\left(1 + \frac{1}{n} \right)^{n}} \cdot \frac{r-1}{n}.
\]
Since the function $\left(1 + \frac{1}{x} \right)^{x}$ is monotonically increasing,
\[
\frac{\left(1 + \frac{1}{r-1} \right)^{r-1}}{\left(1 + \frac{1}{n} \right)^{n}} \cdot \frac{r-1}{n} < \frac{r-1}{n}.
\]
Now, as
\[
\frac{r^{r - 1}}{(r-1)^{r-2}} \cdot \frac{5 (n-1)^{n-1}}{(n+1)^{n+1}} = 5 \frac{\left(1 + \frac{1}{r-1} \right)^{r-1}}{\left(1 + \frac{2}{n-1} \right)^{n-1}} \cdot \frac{r-1}{(n+1)^{2}}
\]
and as the function $\left(1 + \frac{1}{x} \right)^{x}$ is also monotonically increasing,
\[
5 \frac{\left(1 + \frac{1}{r-1} \right)^{r-1}}{\left(1 + \frac{2}{n-1} \right)^{n-1}} \cdot \frac{r-1}{(n+1)^{2}} \leqslant \frac{5}{2} \frac{1}{n+1},
\]
because $(n,r)\not\in\widetilde{\mathbb{D}}$ and therefore,
$(n,r)\not\in\mathbb{D}_{1}$, from which we obtain $ \left(1+ \frac{2}{n-1} \right)^{n-1}
> 2 e$. This allows us to rewrite \eqref{E:sum233} as follows:
\begin{multline*}
f(n+1, r) \leqslant \frac{C (n+1)^{n+1}}{r^{r -1}} \left[ \frac{r-1}{n} +
\frac{5}{2} \frac{1}{n+1}\right] \leqslant \frac{C (n+1)^{n+1}}{r^{r -1}} \cdot \frac{r + 3/2}{n} \leqslant  \frac{C (n+1)^{n+1}}{r^{r -1}},
\end{multline*}
where the last inequality follows from the fact that
$(n,r)\not\in\widetilde{\mathbb{D}}$ and therefore,
$(n,r)\not\in\mathbb{D}_{3}$, that is to say, $r \leqslant n-1$. This concludes the proof.
\end{proof}

We now turn to proving Theorem~\ref{thm1}.
\begin{proof}
Let us define the functions
\[
m(n, x, y) :=\lim_{t\to \infty}\frac{m_{n}(t,x,y)}{m_{1}^{n}(t,x,y)} =\frac{C_{n} (x, y)}{C_{1}^{n} (x, y)},\quad
m(n, x) :=\lim_{t\to \infty}\frac{m_{n}(t,x)}{m_{1}^{n}(t,x)}= \frac{C_{n} (x)}{C_{1}^{n} (x)};
\]
as follows from Theorem~\ref{thm06} and $G_{\lambda}(x,y)$ being positive, these definitions are sound. Corollary~\ref{cor6} yields
\[
m(n, x, y) = m(n, x) = \frac{C_{n}(x)}{C_{1}^{n}(x)} = \frac{C_{n}(x, y)}{C_{1}^{n}(x, y)}.
\]

From these equalities and the asymptotic equivalences~\eqref{E:mainmom} we obtain the equalities~\eqref{eq1} in Theorem~\ref{thm1} in terms of moment convergence of the random variables $\xi(y)=\psi(y)\xi$ and $\xi$.

For the random variables $\xi(y)$ and $\xi$ to be uniquely defined by their
moments, it suffices to demonstrate, as was shown in \cite{YarBRW:e}, that
Carleman's criterion
\begin{equation}\label{E:Carl}
\sum_{n=1}^{\infty} m(n, x, y)^{-1/2n} = \infty,\qquad \sum_{n=1}^{\infty} m(n, x)^{-1/2n} = \infty
\end{equation}
holds.
We establish below that the series for the moments $m(n, x)$ diverges and that, therefore, said moments define the random variable $\xi$ uniquely; the statement concerning $\xi(y)$ and its moments can be proved in much the same manner.

Since  $\beta^{(r)}_{j} = O(r! \,r^{r-1})$, there is such a constant $D$ that
for all $r \geqslant 2$ and $j = 1, \ldots, N$ the inequality
$\beta^{(r)}_{j} < D r! \,r^{r-1}$ holds. We assume without loss of
generality that for all $n$
\[
C_{n}(x) \leqslant \max_{j = 1, \ldots, N}\left(C_{n}(x_{j}) \right) = \left(C_{n}(x_{1}) \right).
\]

Let
\[
\gamma := 2 N \cdot C \cdot D \cdot E\cdot\frac{\lambda_{0}\beta_{2}}{2}\cdot C_{1}^{2}(x_{1}),
\]
where $C$ is as defined in Lemma~\ref{lem7}, and the constant $E$ is such that
$C_{n}(x_{1}) \leqslant \gamma^{n-1} n! \,n^{n}$ for $n \leqslant
\max\{n_{*}, 2 \}$, where $n_{*}$ is defined in Theorem~\ref{thm06}.

Let us show by induction that
\[
C_{n}(x) \leqslant C_{n}(x_{1}) \leqslant \gamma^{n-1} n! \,n^{n}.
\]

The induction basis for $n = 1$ is valid due to the choice of $C$. To prove the inductive step, we will demonstrate that
\[
C_{n+1}(x) \leqslant C_{n+1}(x_{1}) \leqslant \gamma^{n} (n+1)! \,(n+1)^{n+1}.
\]
It follows from the formula for $C_{n+1}(x_{1})$ and the estimate for
$D_{n}^{(j)}(x)$ from Theorem~\ref{thm06} that
\[
C_{n+1}(x_{1}) \leqslant
\sum_{j=1}^{N}\sum_{r=2}^{n+1} \frac{\beta^{(j)}_{r}}{r!} \sum_{\substack{i_{1}, \ldots, i_{r} > 0 \\
i_{1} + \cdots + i_{r} = n+1}} \frac{(n+1)!}{i_{1}! \cdots i_{r}!} C_{i_{1}}(x_{1})\cdots C_{i_{r}}(x_{1}) \frac{2}{\lambda_{0}(n+1)}.
\]
By the induction hypothesis
\[
\frac{(n+1)!}{i_{1}! \cdots i_{r}!} C_{i_{1}}(0) \cdots C_{i_{r}}(0) \leqslant \gamma^{n+1-r} (n+1)! i_{1}^{i_{1}} \cdots i_{r}^{i_{r}};
\]
which, added to the fact that $\beta_{j}^{(r)} < D r! \,r^{r-1}$ and $\gamma^{n+1-r} \leqslant
\gamma^{n-1}$, yields
\begin{multline*}
\sum_{j=1}^{N} \sum_{r=2}^{n+1} \frac{\beta_{j}^{(r)}}{r!} \sum_{\substack{i_{1}, \ldots, i_{r} > 0 \\
i_{1} + \cdots + i_{r} = n+1}} \frac{(n+1)!}{i_{1}! \cdots i_{r}!} C_{i_{1}}(x_{1}) \cdots C_{i_{r}}(x_{1}) \leqslant\\
\leqslant N\gamma^{n-1} D (n+1)! \sum_{r=2}^{n+1}  r^{r-1} \sum_{\substack{i_{1}, \ldots, i_{r} > 0 \\ i_{1} + \cdots + i_{r} = n+1}} i_{1}^{i_{1}} \cdots i_{r}^{i_{r}} =\\
= N \gamma^{n-1}\,D (n+1)! \sum_{r=2}^{n+1}  r^{r-1} f(n+1, r).
\end{multline*}
We infer from Lemma~\ref{lem7} that
\begin{multline*}
N \gamma^{n-1}\,D (n+1)! \sum_{r=2}^{n+1}  r^{r-1} f(n+1, r) \leqslant N \gamma^{n-1} (n+1)!\,D\cdot C \sum_{r=2}^{n+1} (n+1)^{n+1} \leqslant\\\leqslant N \gamma^{n-1} D \cdot C (n+1)! (n+1)^{n+2}.
\end{multline*}
Therefore, by referring to the definition of $\gamma$ we obtain
\[
C_{n+1}(x) \leqslant  \gamma^{n} (n+1)! (n+1)^{n+1},
\]
which completes the proof of the induction step.

Finally, since $n! \leqslant \left(\frac{n+1}{2} \right)^{n}$,
$C_{n}(x) \leqslant \frac{\gamma^{n}}{2^{n}} (n+1)^{2n}$. Thus,
\[
m(n, x) = \frac{C_{n}(x)}{C_{1}^{n}(x)} \leqslant \left(\frac{\gamma}{2 C_{1}(x)} \right)^{n} (n+1)^{2n},
\]
from which it follows that
\[
\sum_{n=1}^{\infty} m(n, x)^{-1/2n} \geqslant \sqrt{\frac{2 C_{1}(x)}{\gamma}} \sum_{n=1}^{\infty} \frac{1}{n+1} = \infty.
\]
Thus the condition~\eqref{E:Carl} holds, and the corresponding Stieltjes moment problem for the moments $m(n,x)$ has a unique solution~\cite[Th.~1.11]{ShT:e}, and therefore the equalities~\eqref{eq1} hold in terms of convergence in distribution. This completes the proof of Theorem~\ref{thm1}.
\end{proof}

\textbf{Acknowledgements.} The research was supported by the Russian
Foundation for Basic Research, project no. 17-01-00468.

%

\end{document}